\documentclass[11pt,a4paper]{article}
\usepackage{amsmath}
\usepackage{amssymb}
\usepackage{amsfonts}
\usepackage{amsthm}
\usepackage{relsize}
\usepackage{setspace}
\usepackage{geometry}
\usepackage{url}
\usepackage{enumerate}

\newtheorem{thm}{Theorem}[section]
\newtheorem{lem}[thm]{Lemma}
\newtheorem{prop}[thm]{Proposition}
\newtheorem{cor}[thm]{Corollary}

\newtheorem*{thma}{Theorem A}
\newtheorem*{thmb}{Theorem B}

\theoremstyle{definition}
\newtheorem{defn}[thm]{Definition}
\newtheorem{que}{Question}

\newcommand{\tF}{\mathrm{(F)}}
\newcommand{\Ilhd}{I^\lhd}

\newcommand{\End}{\mathrm{End}}
\newcommand{\Hom}{\mathrm{Hom}}
\newcommand{\Aut}{\mathrm{Aut}}

\newcommand{\Core}{\mathrm{Core}}

\newcommand{\Con}{\mathrm{Con}}

\newcommand{\bN}{\mathbb{N}}

\newcommand{\bZ}{\mathbb{Z}}

\newcommand{\mcA}{\mathcal{A}}

\newcommand{\mcC}{\mathcal{C}}

\newcommand{\mcF}{\mathcal{F}}

\newcommand{\mcK}{\mathcal{K}}

\newcommand{\mcR}{\mathcal{R}}

\begin{document}

\title{On endomorphisms of profinite groups}

\author{Colin D. Reid\\
Universit\'{e} catholique de Louvain\\
Institut de Recherche en Math\'{e}matiques et Physique (IRMP)\\
Chemin du Cyclotron 2, 1348 Louvain-la-Neuve\\
Belgium\\
colin@reidit.net}

\maketitle

\abstract{We obtain some general restrictions on the continuous endomorphisms of a profinite group $G$ under the assumption that $G$ has only finitely many open subgroups of each index (an assumption which automatically holds, for instance, if $G$ is finitely generated).  In particular, given such a group $G$ and a continuous endomorphism $\phi$ we obtain a semidirect decomposition of $G$ into a `contracting' normal subgroup and a complement on which $\phi$ induces an automorphism; both the normal subgroup and the complement are closed.  If $G$ is isomorphic to a proper open subgroup of itself, we show that $G$ has an infinite abelian normal pro-$p$ subgroup.}

\section{Introduction}

Let $G$ be a finitely generated residually finite group and let $\phi$ be an endomorphism.  It is a well-known theorem of Mal'cev that $G$ is Hopfian, that is, if $\phi$ is surjective, then it is an isomorphism.  On the other hand, there can certainly be injective endomorphisms of $G$ that are not surjective, such as the map $x \mapsto 2x$ for $G = \bZ$.

In this paper, we will focus on profinite groups, and all homomorphisms under consideration are understood to be continuous.  Here the obvious analogue of `finitely generated' is `topologically finitely generated', but in fact it is appropriate to assume a somewhat weaker property:

\begin{defn}Let $G$ be a profinite group.  $G$ is \emph{of type $\tF$} (or an \emph{$\tF$-group}) if $G$ has finitely many open subgroups of index $n$ for every integer $n$.  Equivalently, if we let $\Ilhd_n(G)$ be the intersection of all open normal subgroups of index at most $n$ in $G$, then $\Ilhd_n(G)$ is open in $G$ for all $n$.\end{defn}

Finitely-generated profinite groups are necessarily of type $\tF$ (see for instance \cite{RZ} Proposition 2.5.1.).  In the other direction, a pronilpotent group $G$ is of type $\tF$ if and only if its Sylow subgroups are all finitely generated, but this does not require any overall bound on the number of generators across different primes; thus one can easily construct examples of infinitely generated pronilpotent $\tF$-groups.  In addition, the class of $\tF$-groups includes all just infinite profinite groups, which need not be finitely generated even if they are hereditarily just infinite (see for instance \cite{WilJI}).

Profinite groups of type $\tF$ are Hopfian as topological groups (see \cite{RZ} Proposition 2.5.2.), but not co-Hopfian in general.  Given a profinite group $G$, the question of whether $G$ has any proper \emph{open} subgroups isomorphic to itself is particularly interesting in the theory of totally disconnected, locally compact groups.

Given an endomorphism $\phi$ of $G$, one can define two subgroups of $G$ which measure the extent to which positive powers of $\phi$ fail to be automorphisms:
\[ \Con(\phi) = \{ x \in G \mid \lim_{n \rightarrow +\infty} x^{\phi^n} = 1 \}; \;  \phi_+(G) = \bigcap_{n \ge 0} \phi^n(G).\]

\paragraph{Example}Let $G = H \rtimes K$, where $H$ is the additive group of $\bZ_p$ and $K$ is the group of units of $\bZ_p$, acting by multiplication.  Let $\phi$ be the endomorphism of $G$ that centralises $K$ and acts on $H$ as multiplication by $p$.  Then $\Con(\phi) = H$ and $\phi_+(G) = K$.

Our first main result shows that this splitting as a semidirect product is a general phenomenon for profinite groups of type $\tF$.

\begin{thma}Let $G$ be a profinite group of type $\tF$ and let $\phi$ be an endomorphism of $G$.  Then $G = \Con(\phi) \rtimes \phi_+(G)$ where $\Con(\phi)$ and $\phi_+(G)$ are both closed subgroups.  The restriction of $\phi$ to $\phi_+(G)$ is an automorphism, and we have $\phi^k(\Con(\phi)) = \Con(\phi) \cap \phi^k(G)$ for all $k \ge 0$.\end{thma}

Of course, any profinite group (of type $\tF$) can appear as $\Con(\phi)$ if $\phi$ is the zero endomorphism.  More interesting is the case of \emph{open self-embeddings}, that is, injective endomorphisms $\phi$ of $G$ whose image is open.  In this case one has an associated ascending HNN extension $L = G \ast_\phi$ of $G$, and the topology of $G$ naturally extends to a group topology on $L$ with $G$ as an open subgroup.  In effect, we are therefore considering special cases of the automorphisms and associated contraction groups as considered by Baumgartner, Gl\"{o}ckner and Willis in \cite{BW} and \cite{GW}.  If the images of $G$ under the non-negative powers of $\phi$ have trivial intersection, then $G \ast_\phi$ is a contraction group in the sense of \cite{GW}, which gives a good description of the structure of such groups.  In the special case under consideration here, we can eliminate the case occurring in (\cite{GW}, Theorem A) in which $G$ contains a Cartesian product of non-abelian finite simple groups.

\begin{defn}Given profinite groups $G$ and $H$, write $\Hom_o(G,H)$ for the set of injective homomorphisms from $G$ to $H$ with image open in $H$, and define $\End_o(G) := \Hom_o(G,G)$.  Given $\Lambda \subseteq \End_o(G)$, define
\[ O_\Lambda(G) = \overline{\langle \Con(\lambda) \mid \lambda \in \Lambda \rangle}.\]\end{defn}

\begin{thmb}Let $G$ be a profinite group of type $\tF$ and let $\Lambda \subseteq \End_o(G)$.

\begin{enumerate}[(i)]
\item The group $O_\Lambda(G)$ is pronilpotent.
\item Suppose that the set $\{\lambda(G) \mid \lambda \in \Lambda\}$ forms a base of neighbourhoods of the identity.  Then $G = O_\Lambda(G)$, so $G$ is pronilpotent.\end{enumerate}\end{thmb}

When combined with Theorems A and B of \cite{GW}, Theorem B (i) above can be refined to the following:

\begin{cor}\label{solcor}Let $G$ be a profinite group of type $\tF$ and let $\phi$ be an open self-embedding of $G$.  Then $\Con(\phi)$ is an open subgroup of a direct product of finitely many uniform (in particular, torsion-free) pro-$p$ groups that are nilpotent and a pronilpotent soluble group of finite exponent.\end{cor}

As a result we obtain the following, which generalises Theorem E (i) of \cite{Rei}.

\begin{cor}Let $G$ be a profinite group of type $\tF$.  Suppose that there is a proper open subgroup of $G$ that is isomorphic to $G$ itself.  Then $G$ has an infinite abelian pro-$p$ normal subgroup.\end{cor}

Theorems A and B will be proved in somewhat greater generality, requiring some more technical definitions to state.  We will also prove some facts about open self-embeddings in the class of all profinite groups.  At this point it seems reasonable to ask the following:

\begin{que}Let $G$ be a finitely generated profinite group, and let $\phi$ be an open self-embedding of $G$.  Is $\Con(\phi)$ necessarily nilpotent?\end{que}

\section{Preliminaries and general remarks}

The following standard compactness argument will be used in several places.

\begin{lem}\label{obchain}Let $G$ be a compact topological group, and let $O$ be an open neighbourhood of $1$ in $G$.  Let $\mcK$ be a set of closed subgroups of $G$ such that $\bigcap \mcA \not\subseteq O$ for every finite subset $\mcA$ of $\mcK$.  Let $N=\bigcap \mcK$.  Then $N \not\subseteq O$; in particular, $N$ is non-trivial.\end{lem}

\begin{proof}Given $K \in \mcK$, let $C_K = K \cap (G \setminus O)$.  Then $\{C_K \mid K \in \mcK\}$ is a set of closed subsets of $G$ whose finite subsets have non-empty intersection.  Since $G$ is compact, it follows that the intersection $N \cap (G \setminus O)$ of all the sets $C_K$ is non-empty.  Hence $N \not\subseteq O$.\end{proof}

\begin{cor}\label{obcor}Let $G$ be a profinite group, and let $H$ be a closed subgroup of $G$.  Let $\mcK$ be a set of closed subgroups of $G$ such that $\overline{\langle \bigcap\mcA,H \rangle} = G$ for every finite subset $\mcA$ of $\mcK$.  Let $K=\bigcap \mcK$.  Then $\overline{\langle K, H \rangle}=G$.\end{cor}

\begin{proof}Let $L = \overline{\langle K, H \rangle}$, and suppose $L<G$.  Then $L \leq M$ for some proper open subgroup $M$ of $G$, since any closed subgroup of a profinite group is an intersection of open subgroups.  Then $H \leq M$, and it follows from Lemma \ref{obchain} that there is some finite subset $\mcA$ of $\mcK$ such that $\bigcap \mcA \leq M$.  But then $\overline{\langle \bigcap \mcA,H \rangle} \leq M < G$, a contradiction.\end{proof}

The following trivial observation will be used without further comment.

\begin{lem}\label{normend}Let $G$ be a profinite group, let $K \unlhd_c G$ and let $\phi$ be an endomorphism of $G$.  Then $\phi^{-1}(K) \unlhd_c G$.\end{lem}

When considering the existence of open embeddings between profinite groups, many questions reduce to those about profinite groups involving finitely many primes, thanks to the following:

\begin{prop}\label{fewprimes}Let $G$ and $H$ be profinite groups, and suppose there is a continuous injective homomorphism $\phi$ from $G$ to an open subgroup $K$ of $H$.  Given a set of primes $\pi$, write $O^\pi(G)$ for the intersection of all closed normal subgroups $L$ of $G$ such that $G/L$ is a pro-$\pi$ group.  Then there is a finite set of primes $\pi$ such that for every set of primes $\pi^*$ containing $\pi$, there is an injective homomorphism from $G/O^{\pi^*}(G)$ to $H/O^{\pi^*}(H)$ whose image is $KO^{\pi^*}(H)/O^{\pi^*}(H)$.\end{prop}

\begin{lem}\label{lambdareslem}Let $G$ be a profinite group, and let $H \leq_o G$.  Then there is a finite set of primes $\pi$ such that $O^{\pi^*}(H) = O^{\pi^*}(G)$, for any set of primes $\pi^*$ containing $\pi$.\end{lem}

\begin{proof}Let $K \unlhd_o G$, let $\pi$ be the set of primes dividing $|G:K|$ and let $\pi^* \supseteq \pi$.  Then $O^{\pi^*}(G) \le K$ since $G/K$ is a $\pi$-group and hence a $\pi^*$-group.  Moreover $O^{\pi^*}(K) \le O^{\pi^*}(G)$ since $K/O^{\pi^*}(G)$ is contained in $G/O^{\pi^*}(G)$ and is therefore pro-$\pi^*$.  On the other hand $G/O^{\pi^*}(K)$ is pro-$\pi^*$, being the extension of a pro-$\pi^*$ group by another pro-$\pi^*$ group, so in fact $O^{\pi^*}(K) = O^{\pi^*}(G)$.

We apply this argument inside $G$ and inside $H$ to obtain $O^{\pi^*}(G) = O^{\pi^*}(K) = O^{\pi^*}(H)$ where $K$ is the core of $H$ in $G$, so the required finite set of primes $\pi$ is given by the prime divisors of $|G:\Core_G(H)|$.\end{proof}

\begin{proof}[Proof of Proposition \ref{fewprimes}]By Lemma \ref{lambdareslem}, there is a finite set of primes $\pi$ such that $O^{\pi^*}(\phi(G)) = O^{\pi^*}(H)$, for any set of primes $\pi^*$ containing $\pi$.  It follows that there is a well-defined and unique homomorphism $\psi$ from $G/O^{\pi^*}(G)$ to $H/O^{\pi^*}(H)$ such that
\[ \psi(xO^{\pi^*}(G))= \phi(x) O^{\pi^*}(H) \]
for all $x \in G$.  We also have $\psi(G/O^{\pi^*}(G)) = \phi(G) O^{\pi^*}(H) \leq_o H$, and
\[ \ker(\psi) = \{ xO^{\pi^*}(G) \mid x \in G, \phi(x) \in O^{\pi^*}(\phi(G))\} = 1.\]
Hence $\psi$ is injective, as required.\end{proof}

Preimages of endomorphisms are well-behaved with respect to open subgroups.

\begin{lem}\label{shrinkind}Let $G$ be a profinite group, let $K \le_o G$ and let $\phi$ be an endomorphism of $G$.  Then $|G:\phi^{-1}(K)| \le |G:K|$.  If $|G:\phi^{-1}(K)| = |G:K|$ then $G = \phi(G)K$.\end{lem}

\begin{proof}Let $L = \phi^{-1}(K)$.  Note that $|\phi(G):K \cap \phi(G)| \le |G:K|$, with equality if and only if $\phi(G)$ contains a transversal of $K$ in $G$, that is, if and only if $G = \phi(G)K$. and that $\phi$ induces a surjective map from left cosets of $L$ in $G$ to left cosets of $K \cap \phi(G)$ in $\phi(G)$.  Suppose $|G:L|>|\phi(G):K \cap \phi(G)|$.  Then by the pigeon-hole principle, there are distinct cosets $xL$ and $yL$ of $L$ in $G$ such that \; $\phi(x)(K \cap \phi(G)) =\phi(y)(K \cap \phi(G))$.  But then $\phi(xy^{-1}) = \phi(x)\phi(y^{-1}) \in K$, so in fact $xy^{-1} \in L$, a contradiction.  Hence $|G:L| \le |\phi(G):K \cap \phi(G)| \le |G:K|$, and if $|G:L|=|G:K|$ then $G = \phi(G)K$.\end{proof}

To determine whether every open subgroup of a profinite group $G$ contains an open subgroup isomorphic to $G$, it suffices to consider maximal open normal subgroups.

\begin{prop}Let $G$ be a profinite group and let $H$ be an open subgroup of $G$.  Suppose that $\Hom_o(G,H) = \emptyset$.  Then $\Hom_o(G,K) = \emptyset$ for some open normal subgroup $K$ of $G$ such that $G/K$ is simple.\end{prop}

\begin{proof}By replacing $H$ by its core in $G$, we may assume $H$ is normal in $G$; moreover, we may assume $|G:H|$ realises the minimum value of the set
\[ \{ |G:N| \mid N \unlhd_o G, \Hom_o(G,N)=\emptyset\}.\]
Clearly $H < G$, so there is a normal subgroup $K$ containing $H$ such that $G/K$ is simple.  Suppose there exists $\phi \in \Hom_o(G,K)$; let $L = \phi^{-1}(H)$.  Then $\phi(G)H \le K < G$, so $|G:L| < |G:H|$ by Lemma \ref{shrinkind}.  By the minimality of $|G:H|$, there must be some $\psi \in \Hom_o(G,L)$.  But then $\phi|_L\psi \in \Hom_o(G,H)$, a contradiction.  Hence $\Hom_o(G,K)=\emptyset$.\end{proof}

As an extreme case, one could consider profinite groups $G$ in which \emph{every} open subgroup $H$ of $G$ is isomorphic to $G$, or at least there is some (open) normal subgroup of $G$ contained in $H$ that is isomorphic to $G$.  However, there are no examples of this behaviour for groups of type $\tF$ aside from the obvious ones.

\begin{prop}\label{manyiso}Let $G$ be a profinite group of type $\tF$.  Then the following are equivalent:
\begin{enumerate}[(i)]
\item Every open subgroup of $G$ contains a normal subgroup $N$ of $G$ such that $N \cong G$.
\item $G$ is of the form $\prod_p \bZ^{n_p}_p$ for non-negative integers $n_p$, where $p$ ranges over a set of prime numbers.  In particular, every open subgroup of $G$ is isomorphic to $G$.\end{enumerate}\end{prop}

\begin{proof}It is clear that (ii) implies (i).  Suppose (i) holds; let $\mcC = \{\Ilhd_n(G) \mid n \in \bN\}$.

We claim first that $G$ is abelian.  Since $\mcC$ forms a base of neighbourhoods of the identity, it suffices to show that $G/H$ is abelian for all $H \in \mcC$.  Fix $H$; let $k = |\Aut(G/H)|$ and let $K = \Ilhd_k(G)$.  Then by our assumption, there is an injective endomorphism $\phi$ of $G$ such that $\phi(G)$ is normal in $G$ and $\phi(G) \le K$.  Since $\phi(H)$ is a characteristic subgroup of $\phi(G)$, we see that $G$ acts on $\phi(G)/\phi(H)$ by conjugation as a subgroup of $\Aut(\phi(G)/\phi(H)) \cong \Aut(G/H)$, so $|G:C_G(\phi(G)/\phi(H))| \le K$, and hence $K \le C_G(\phi(G)/\phi(H))$.  In particular, $\phi(G)$ acts trivially on $\phi(G)/\phi(H)$, so $\phi(G)/\phi(H)$ is abelian, and thus $G/H$ is abelian.

Thus $G$ is a Cartesian product of its Sylow subgroups $S_p$, each of which is abelian of type $\tF$.  This forces $S_p = F_p \times \bZ^{n_p}_p$ where $F_p$ is finite; in fact $F_p$ must be trivial for (i) to hold, by considering the $p$-torsion subgroups of open subgroups of $G$.  Thus $G$ is of the form described in (ii).\end{proof}

\section{Well-behaved groups and endomorphisms}

We now define a somewhat technical property of a set of endomorphisms acting on a profinite group.  This property holds for all endomorphisms of $\tF$-groups, but is additionally retained on passage to suitable invariant subgroups.

\begin{defn}\label{regdef}Let $G$ be a profinite group and let $\Lambda$ be a set of endomorphisms of $G$.  Say a subgroup $H$ of $G$ is \emph{$\Lambda$-invariant} if $\lambda(H) \le H$ for all $\lambda \in \Lambda$.  Say $\Lambda$ is \emph{stable} on $G$ if the set of all $\Lambda$-invariant open subgroups of $G$ is a base of neighbourhoods of the identity.

Given a set $\Omega$ of automorphisms of $G$, let $I^\Omega_n(G)$ be the intersection of all $\Omega$-invariant open normal subgroups of $G$ of index at most $n$.  Say $\Omega$ \emph{regulates} $\Lambda$ (on $G$) if the following conditions are satisfied:

\begin{enumerate}[(a)]
\item $\lambda\Omega = \Omega\lambda$ for all $\lambda \in \Lambda$.
\item $I^\Omega_n(G)$ is open in $G$ for all $n$.
\item The set $\{I^\Omega_n(G) \mid n \in \bN\}$ is a base of neighbourhoods of the identity.\end{enumerate}

Say that $\Lambda$ is \emph{regulated} on $G$ if a set of $\Omega$ of automorphisms exists such that $\Omega$ regulates $\Lambda$.  Note that subsets of regulated sets are regulated.\end{defn}

If $G$ is of type $\tF$ then the set of all endomorphisms is regulated (for instance by the empty set).

\begin{prop}\label{tfrelstab}Let $G$ be a profinite group and let $\Lambda$ be a set of endomorphisms of $G$.  Suppose that $\Lambda$ is regulated by $\Omega \subseteq \Aut(G)$ on $G$.
\begin{enumerate}[(i)]
\item If $K$ is an $\Omega$-invariant open normal subgroup of $G$ and $\lambda \in \Lambda$, then $\lambda^{-1}(K)$ is an $\Omega$-invariant open normal subgroup of $G$.  As a result, $I^\Omega_n(G)$ is $\Lambda$-invariant for all $n$, so $\Lambda$ is stable on $G$.
\item Suppose that $G$ is of the form $N \rtimes H$ where both $N$ and $H$ are $\Lambda \cup \Omega$-invariant, and where every element of $\Lambda$ induces a surjective map on $H$.  Let $\Psi$ be the set of automorphisms of $N$ induced by the conjugation action of $H$.  Then $\Xi = \Omega \cup \Psi$ regulates $\Lambda$ acting on $N$.\end{enumerate}\end{prop}

\begin{proof}(i) Let $\lambda \in \Lambda$ and $\omega \in \Omega$.  Then $\lambda\omega = \omega'\lambda$ for some $\omega' \in \Omega$, so
\[ \lambda\omega(\lambda^{-1}(K)) = \omega'\lambda(\lambda^{-1}(K)) = \omega'(K \cap \lambda(G)) \le \omega'(K) = K\]
and thus $\omega(\lambda^{-1}(K)) \le \lambda^{-1}(K)$, so $\lambda^{-1}(K)$ is $\Omega$-invariant.  We have $|G:\lambda^{-1}(K)| \le |G:K|$ by Lemma \ref{shrinkind}.  It follows that $\lambda^{-1}(I^\Omega_n(G))$ is an intersection of $\Omega$-invariant open normal subgroups of $G$ of index at most $n$, so $\lambda^{-1}(I^\Omega_n(G)) \ge I^\Omega_n(G)$, that is $\lambda(I^\Omega_n(G)) \le I^\Omega_n(G)$.

(ii) We must check that each of the conditions (a)--(c) are satisfied by $\Xi$.

Let $\lambda \in \Lambda$.  We have $\lambda\Omega = \Omega\lambda$ by assumption.  Suppose $h \in H$ induces $\psi_h \in \Psi$ by conjugation.  Then for all $x \in N$:
\[ \lambda\psi_h(x) = \lambda(hxh^{-1}) = \lambda(h)\lambda(x)\lambda(h)^{-1} = \psi_{\lambda(h)}\lambda(x),\]
so $\lambda\psi_h = \psi_{\lambda(h)}\lambda$.  Similarly $\psi_h\lambda = \lambda\psi_k$, where $k$ is an element of $H$ such that $\lambda(k) = h$.  This proves condition (a).

Let $K$ be a $\Xi$-invariant open normal subgroup of $N$ of index at most $n$.  Then $K$ is normalised by $H$, so $K \unlhd G$, and $KH$ is an open $\Omega$-invariant subgroup of $G$, so the core of $KH$ in $G$ is an open normal $\Omega$-invariant subgroup of $G$.  Thus $KH \ge I^\Omega_{n!}(G)$, which ensures that $K$ contains $I^\Omega_{n!}(G) \cap N$.  Thus $I^\Xi_n(N)$ contains the open subgroup $I^\Omega_{n!}(G) \cap N$ of $N$.  Moreover, we see that $I^\Xi_{t(n)}(N) \le I^\Omega_n(G) \cap N$ for all $n \in \bN$, where $t(n) = |G:I^\Omega_n(G)|$, so the set $\{I^\Xi_n(N) \mid n \in \bN\}$ has trivial intersection.  This proves conditions (b) and (c).\end{proof}

We can also generalise the hypotheses of Theorem A by considering contraction for more general semigroups of endomorphisms, in place of the cyclic monoid $\{\phi^k \mid k \ge 0\}$.

\begin{defn}Let $G$ be a profinite group and let $\Lambda$ be a semigroup of endomorphisms of $G$ acting on the left.  Write $\Lambda_\cap(G)$ for $\bigcap_{\lambda \in \Lambda} \lambda(G)$.  Let $\mcF$ consist of all subsets of $\Lambda$ which contain $\bigcap_{\xi \in \Xi} \Lambda\xi$ for some finite subset $\Xi$ of $\Lambda$.   Given a closed subgroup $K$ of $G$, let $\Con(\Lambda,K)$ be the set of elements $x \in G$ such that for all open sets $O \supseteq K$ there is some $\Sigma \in \mcF$ such that $\sigma x \in O$ for all $\sigma \in \Sigma$; define $\Con(\Lambda) = \Con(\Lambda,1)$.   (One can regard $\Con(\Lambda)$ as the set on which $\Lambda$ converges pointwise to the zero endomorphism.)  For these definitions to be useful, we are particularly interested in those semigroups $\Lambda$ which satisfy following conditions:

(I) For every finite subset $\Xi$ of $\Lambda$, the intersection $\bigcap_{\xi \in \Xi} \Lambda\xi$ is non-empty.

(II) For every finite subset $\Xi$ of $\Lambda$, there is some $\lambda \in \Lambda$ such that $\lambda(G) \le \Xi_\cap(G)$.

Notice that conditions (I) and (II) are automatically satisfied if $\Lambda$ is a commutative semigroup.  In the simplest case, when $\Lambda$ is the set of non-negative powers of a single endomorphism $\phi$, we see that $\Con(\Lambda) = \Con(\phi)$ and $\Lambda_\cap(G) = \phi_+(G)$.\end{defn}

\begin{lem}\label{conlem}Let $G$ be a profinite group and let $\Lambda$ be a semigroup of endomorphisms of $G$.  Let $K$ be a closed subgroup of $G$.

(i) The set $\Con(\Lambda,K)$ is a subgroup of $G$, and if $K \unlhd G$ then $\Con(\Lambda,K) \unlhd G$.

(ii) Given $\lambda \in \Lambda$ and $x \in G$, then $\lambda(\Con(\Lambda,K)) \ge \Con(\Lambda,K) \cap \lambda(G)$ and $\ker \lambda \le \Con(\Lambda,K)$.\end{lem}

\begin{proof}(i) Given $x,y \in \Con(\Lambda,K)$, there exist $\Sigma_1,\Sigma_2 \in \mcF$ such that $x \in  \bigcap_{\sigma \in \Sigma_1} \sigma^{-1}(K)$ and $y \in  \bigcap_{\sigma \in \Sigma_2} \sigma^{-1}(K)$.  Since $\mcF$ is a filter we also have $\Sigma_3 = \Sigma_1 \cap \Sigma_2 \in \mcF$, and since the preimages $\sigma^{-1}(K)$ are always subgroups of $G$, it follows that $xy^{-1} \in \bigcap_{\sigma \in \Sigma_3} \sigma^{-1}(K) \subseteq \Con(\Lambda,K)$.  Hence $\Con(\Lambda,K)$ is a group.

Now suppose $K \unlhd G$, let $x \in \Con(\Lambda,K)$, let $O \unlhd_o G$ such that $K \le O$.  Then there is some $\Sigma \in \mcF$ such that $\sigma(x) \in O$ for all $\sigma \in \Sigma$.  This implies that $\sigma(y^{-1}xy) \in O$ for all $\sigma \in \Sigma$ and $y \in G$.  The fact that $\Con(\Lambda,K)$ is normal now follows from the fact that $K$ is the intersection of the open normal subgroups of $G$ that contain it.

(ii) Let $x \in G$ and let $y = \lambda(x)$; suppose $y \in \Con(\Lambda,K)$.  Then for all open sets $O$ containing $K$, there is some $\Sigma \in \mcF$ such that $\sigma \lambda(x) = \sigma(y) \in O$ for all $\sigma \in \Sigma$.  Since $\Sigma\lambda \in \mcF$ it follows that $x \in \Con(\Lambda,K)$.  The case $y=1$ shows that $\ker\lambda \le \Con(\Lambda,K)$, and in general we see that $\lambda(\Con(\Lambda,K)) \ge \Con(\Lambda,K) \cap \lambda(G)$.\end{proof}

\section{Main theorems}

\begin{thm}\label{splitthm}Let $G$ be a countably based profinite group.  Let $\Lambda$ be a semigroup of endomorphisms of $G$ that is stable and satisfies condition (I).  Then $\Con(\Lambda)$ is a closed normal subgroup of $G$ that has trivial intersection with $\Lambda_\cap(G)$.  For any $\lambda \in \Lambda$ we have $G = \Con(\Lambda)\lambda(G)$ and $\Con(\Lambda) \cap \lambda(G) = \lambda(\Con(\Lambda))$.  In particular, $\Lambda$ restricts to a stable semigroup of endomorphisms of $\Con(\Lambda)$ and $\Lambda_\cap(\Con(\Lambda)) = 1$.

If in addition $\Lambda$ satisfies condition (II), then $G = \Con(\Lambda) \rtimes \Lambda_\cap(G)$ and every $\lambda \in \Lambda$ restricts to an automorphism on $\Lambda_\cap(G)$.\end{thm}

\begin{proof}Let $N = \Con(\Lambda)$ and let $H = \Lambda_\cap(G)$. 

Let $K \leq_o G$ and for $\Sigma \in \mcF$ write $R_\Sigma =  \bigcap_{\sigma \in \Sigma} \sigma^{-1}(K)$ and let $\mcR = \{R_\Sigma \mid \Sigma \in \mcF\}$. Then $K$ contains a $\Lambda$-invariant open subgroup $L$ of $G$, so $R_\Sigma \ge L$ for all $\Sigma \in \mcF$.  Hence $\Con(\Lambda,K)$ is an open subgroup of $G$, since $\Con(\Lambda,K) = \bigcup \mcR$.  Thus by Lemma \ref{conlem}, we have 
\[|G:\Con(\Lambda,K)| \ge |\lambda(G): \lambda(G) \cap \Con(\Lambda,K)| \ge |\lambda(G):\lambda(\Con(\Lambda,K))| = |G:\Con(\Lambda,K)|,\]
where all indices are finite.  We see that in fact equality must hold for both inequalities here.

For the first inequality, this ensures $\Con(\Lambda,K)\lambda(G) = G$.  We let $K$ range over a descending chain of open subgroups with trivial intersection; the corresponding groups $\Con(\Lambda,K)$ then also form a descending chain.  Applying Corollary \ref{obcor} then gives $N\lambda(G) = G$.

Given Lemma \ref{conlem} (ii), the equality $|\lambda(G): \lambda(G) \cap \Con(\Lambda,K)| = |\lambda(G):\lambda(\Con(\Lambda,K))|$ implies that $\Con(\Lambda,K) \cap \lambda(G) = \lambda(\Con(\Lambda,K))$ for all open subgroups $K$; this implies that $N \cap \lambda(G) = \lambda(N)$, so $N \cap H = \Lambda_\cap(N)$.

We have $\Con(\Lambda,K) = \bigcup \mcR$; in fact, $\mcR$ is a direct system of open subgroups, since $R_{\Sigma_1 \cap \Sigma_2} \ge \langle R_{\Sigma_1},R_{\Sigma_2}\rangle$, so by the compactness of $\Con(\Lambda)$ we must have $\Con(\Lambda,K) = R_\Sigma$ for some $\Sigma$ depending on $K$.  Condition (I) ensures this set $\Sigma$ is non-empty, in other words there exists $\sigma \in \Lambda$ such that $\sigma(\Con(\Lambda,K)) \le K$, so $\Lambda_\cap(\Con(\Lambda,K)) \le K$.  Since $N$ is the intersection of $\Con(\Lambda,K)$ as $K$ ranges over the open subgroups, it follows that $N$ is closed and $\Lambda_\cap(N) = 1$, and hence $N \cap H = 1$.

Now suppose condition (II) holds.  Then we have $N\Xi_\cap(G) = G$ for all finite subsets $\Xi$ of $\Lambda$, so $G = NH$ by Corollary \ref{obcor} and hence $G = N \rtimes H$.

Fix $\lambda \in \Lambda$; it remains to show that $\lambda$ induces an automorphism on $H$. We have $H \cap \ker\lambda \le H \cap N = 1$, so $\lambda$ induces an isomorphism from $H$ to $\lambda(H)$; since $\lambda\Lambda \subseteq \Lambda$, we also have $\lambda(H) \ge H$.  If $\lambda(H) > H$, then since $G=NH$ we must have some $x \in H \setminus \{1\}$ such that $\lambda(x) \in N$.  But then $x \in N$ as in the proof of Lemma \ref{conlem} (ii), so $x \in N \cap H = 1$, a contradiction.\end{proof}

Theorem \ref{splitthm} and Proposition \ref{tfrelstab} combined imply Theorem A.  They also have the following consequence:

\begin{cor}\label{conreg}Let $G$ be a countably based profinite group.  Let $\Lambda$ be a semigroup of endomorphisms of $G$ that is regulated and satisfies conditions (I) and (II).  Then $\Lambda$ is regulated on $\Con(\Lambda)$.\end{cor}

The following theorem is a more general form of Theorem B.  Corollary \ref{solcor} can of course be generalised in the same way.

\begin{thm}\label{homogthm}Let $G$ be a profinite group and let $\Lambda \subseteq \End_o(G)$.  Suppose that $\Lambda$ is regulated on $G$.

\begin{enumerate}[(i)]
\item The group $O_\Lambda(G)$ is pronilpotent.
\item Suppose that the set $\{\lambda(G) \mid \lambda \in \Lambda\}$ forms a base of neighbourhoods of the identity.  Then $G = O_\Lambda(G)$, so $G$ is pronilpotent.\end{enumerate}\end{thm}

\begin{proof}(i) To show that $O_\Lambda(G)$ is pronilpotent, it suffices to show $\Con(\lambda)$ is pronilpotent for all $\lambda \in \Lambda$, since any profinite group has a unique largest pronilpotent normal subgroup.  Note that the set $\Delta$ of non-negative powers of $\lambda$ is regulated on $G$ and hence also on $\Con(\lambda)$ by Corollary \ref{conreg}.  Thus we may assume $G = \Con(\lambda)$.

Let $k = |G:\lambda(G)|$; we may assume $k > 1$, as otherwise $\lambda$ is an automorphism and $\Con(\lambda)=1$.  Let $\Omega$ be a regulating set for $\Delta$ on $G$, and equip $\Omega$ with the discrete topology.  Define the following series of subgroups of $L = G \rtimes \Omega$:
\[ N_0 = G; \quad N_{i+1} = N_i \cap \Core_{N_i\Omega}((\lambda^{i+1}(G) \cap N_i)\Omega).\]

Let $i \in \bN$.  By construction, $N_i$ is $\Omega$-invariant and we have $N_i \le \lambda^i(G)$ (so in particular $\bigcap N_i = 1$) and $N_{i+1} \unlhd N_i$.

Given $\omega \in \Omega$, note that $\omega\lambda^{i+1} = \lambda^{i+1}\omega'$ for some $\omega' \in \Omega$ (as automorphisms of $G$) by condition (a).  Thus $\lambda^{i+1}(G) \cap N_i$ is normalised by $\Omega$ in $L$, so $(\lambda^{i+1}(G) \cap N_i)\Omega$ is a subgroup of $N_i\Omega$ of index at most $k$, and hence $|N_i:N_{i+1}| = |N_i\Omega:N_{i+1}\Omega| \le k!$.

Now let $M$ be an open subgroup of $G$ that is maximal subject to the conditions that $M$ is $\Omega$-invariant and that there is an $\Omega$-invariant closed subnormal series from $M$ to $L$; let $\Phi^\Omega(G)$ be the intersection of all such subgroups $M$.  Then clearly $M \lhd G$, so $M$ is a maximal $\Omega$-invariant open normal subgroup of $G$.  As $M$ is a proper open subgroup of $G$, there is some $i$ such that $N_i \not\le M$ but $N_{i+1} \le M$.  Thus $MN_i$ is $\Omega$-invariant, and there is an $\Omega$-invariant subnormal series from $MN_i$ to $G$; thus $MN_i = G$ by the maximality of $M$.  This ensures that $|G:M| = |MN_i:MN_{i+1}| \le |N_i:N_{i+1}| \le k!$.  It follows that $\Phi^\Omega(G) \ge I^\Omega_{k!}(G)$, so $\Phi^\Omega(G)$ is an open subgroup of $G$.  Thus there exists $i$ such that $\lambda^i(G) \le \Phi^\Omega(G)$; by replacing $\lambda$ with $\lambda^i$ we may assume $i=1$.

Let $H = \lambda(G)$.  Given $t \in \bN$, let $G_t = I^\Omega_t(G)$ and let $H_t = I^\Omega_t(H)$.  Consider a proper $\Omega$-invariant open normal subgroup $K$ of $G$ such that $|G:K| \le t+1$.  Then $KH$ is a proper open subgroup of $G$, since it is contained in $K\Phi^\Omega(G)$.  In addition, $H \cap K$ is an $\Omega$-invariant open normal subgroup of $H$, and $|H:H \cap K|=|KH:K| < |G:K|$; in particular, $|H:H \cap K| \le t$.  Hence $K$ contains $H_t$.  Since this argument holds for all $t$, and for all proper $\Omega$-invariant open normal subgroups of $G$ of index at most $t+1$, it follows that $G_{t+1} \ge H_t$ for all $t$.  Hence:
\[ |G:H_t|= k|H:H_t| = |G:G_t||G_t:G_{t+1}||G_{t+1}: H_t|.\]
Now $\lambda$ induces an isomorphism from $G$ to $H$ that sends $\Omega$-invariant subgroups to $\Omega$-invariant subgroups (by virtue of condition (a) in Definition \ref{regdef}), so $|H:H_t| = |G:G_t|$ and hence $|G_t:G_{t+1}|$ divides $k$ for all $t$.

Let $n$ be the largest order of $\Aut(F)$, as $F$ ranges over all finite groups of order dividing $k$, and let $R =  I^\Omega_n(G)$. Then for all $t$, the centraliser of $G_t/G_{t+1}$ in $G$ is an $\Omega$-invariant open normal subgroup of $G$ of index at most $n$, and so contains $R$; thus the series $(R \cap G_t)_{t \in \bN}$ is a central series for $R$, and in particular $R$ is pronilpotent.  Moreover $R$ has finite index in $G$, so there is some $i$ such that $\lambda^i(G) \le R$; since $G \cong \lambda^i(G)$, it follows that $G$ is pronilpotent.

(ii) Suppose $O_\Lambda(G) < G$.  Then there is some proper open subgroup $K$ of $G$ containing $O_\Lambda(G)$, and some $\lambda \in \Lambda$ such that $\lambda(G) \le K$ and hence $O_\Lambda(G)\lambda(G) \le K$; in particular, $\Con(\lambda)\lambda(G) \le K$.
However, by Proposition \ref{tfrelstab} and Theorem \ref{splitthm} we have $\Con(\lambda)\lambda(G) = G$, a contradiction.\end{proof}

\paragraph{Acknowledgements}My thanks go to Charles Leedham-Green for helpful discussions on an earlier version of this paper.

\end{document}